\newif\ifpictures
\picturestrue

\documentclass[12pt]{amsart}
\usepackage{amssymb}
\usepackage{amsmath}
\usepackage{dsfont}
\usepackage{amscd}
\usepackage[pdftex]{graphicx}
\usepackage[normalem]{ulem}
\usepackage{hyperref}
\usepackage{tikz}
\usepackage{xcolor}
\usepackage{bbm} 

\numberwithin{equation}{section}
\newtheorem{thm}{Theorem}
\newtheorem{prop}[thm]{Proposition}
\newtheorem{lemma}[thm]{Lemma}
\newtheorem{cor}[thm]{Corollary}

\theoremstyle{definition}

\newtheorem{example}[thm]{Example}
\newtheorem{remark1}[thm]{Remark}
\newtheorem{openproblem1}[thm]{Open problem}
\newtheorem{definition}[thm]{Definition}

\newenvironment{rem}{\begin{remark1}\rm}{\end{remark1}}

\numberwithin{thm}{section}

\newcounter{FNC}[page]
\def\newfootnote#1{{\addtocounter{FNC}{2}$^\fnsymbol{FNC}$%
		\let\thefootnote\relax\footnotetext{$^\fnsymbol{FNC}$#1}}}

\newcommand{\N}{\mathbb{N}}
\newcommand{\Q}{\mathbb{Q}}
\newcommand{\R}{\mathbb{R}}

\newcommand{\Scal}{\mathcal{S}}

\newcommand{\sonc}{\mathrm{SONC}}
\newcommand{\sage}{\mathrm{SAGE}}

\newcommand{\xb}{\mathbf{x}}
\newcommand{\yb}{\mathbf{y}}
\newcommand{\wb}{\mathbf{w}}

\newcommand{\vb}{\mathbf{v}}
\newcommand{\bb}{\mathbf{b}}
\newcommand{\cb}{\mathbf{c}}
\newcommand{\db}{\mathbf{d}}
\newcommand{\ub}{\mathbf{u}}

\newcommand{\cA}{\mathcal{A}}
\newcommand{\cB}{\mathcal{B}}

\newcommand{\Sc}{\mathcal{S}}

\newcommand{\CS}{C_\Sc}

\newcommand{\Po}[1]{P^{\mathrm{odd}}_{#1}}
\newcommand{\Pe}[1]{P^{\mathrm{even}}_{#1}}

\DeclareMathOperator{\conv}{conv}

\DeclareMathOperator{\cone}{cone}

\DeclareMathOperator{\relinter}{relint}

\newcommand{\set}[1]{\{#1\}}

\usepackage{todonotes}

\usepackage{xcolor}

\title{The $\Sc$-cone and a primal-dual view on second-order representability}

\author{Helen Naumann}

\author{Thorsten Theobald}

\address{Helen Naumann, Thorsten Theobald:
Goethe-Universit\"at, FB 12 -- Institut f\"ur Mathematik,
Postfach 11 19 32, D--60054 Frankfurt am Main, Germany}

\email{\{naumann,theobald\}@math.uni-frankfurt.de}

\date{\today}

\begin{document}
	\begin{abstract}	
The $\mathcal{S}$-cone provides a common framework for cones of polynomials
or exponential sums which establish non-negativity upon the arithmetic-geometric inequality, in particular for sums of non-negative circuit polynomials (SONC) or
sums of arithmetic-geometric exponentials (SAGE).
In this paper, we study the $\mathcal{S}$-cone and its dual from the viewpoint of second-order representability. Extending 
results of Averkov and of Wang and Magron on the primal SONC 
cone, we provide explicit generalized second-order descriptions for 
rational $\mathcal{S}$-cones and their duals.
\end{abstract}
	
	\maketitle
	
\section{Introduction}
The question to characterize and to decide whether a polynomial or an exponential 
sum is non-negative occurs in many branches of mathematics and application
areas.
In the development of real algebraic geometry, the connection between
the cone of non-negative polynomials and the cone of sums of squares of 
polynomials plays a prominent role (see, for example, 
\cite{bcr-98,marshall-book,prestel-delzell-2001}). 
If a polynomial
can be written as a sum of squares of polynomials, this provides a certificate
for the non-negativity of the polynomial.
Since the beginning of the current millennium, non-negativity certificates of polynomials
have also seen much interest from the computational point of view and have
strongly advanced the rich connections between real and convex
algebraic geometry as well as polynomial optimization 
(see, for example, \cite{lasserre-book,laurent-2009-survey}).

 Within the research activities on non-negativity certificates in the last years,
 the cones of sums of arithmetic-geometric exponentials
 (SAGE, introduced by Chandrasekaran and Shah \cite{chandrasekaran-shah-2016}) 
 and sums of non-negative circuit polynomials (SONC, introduced by Iliman
 and de Wolff \cite{iliman-dewolff-resmathsci}) have received a lot of attention
 (see, e.g., 
 \cite{averkov-2019,DKdW,forsgard-de-wolff-2019,mcw-2018,
   murray-partial-dual,wang-2018}).
 These cones
 build upon earlier work of Reznick \cite{reznick-1989}. They
 provide non-negativity certificates based on the arithmetic-geometric
 inequality and are particularly useful in the context of sparse polynomials.

 In \cite{knt-2019}, the authors of the current paper and Katth\"an have 
 introduced a common generalization, called the $\mathcal{S}$-cone, which facilitates 
 to study the SAGE cone and the SONC cone within a uniform generalized setting.
 Formally, for two finite disjoint sets $\emptyset \neq \cA \subseteq \R^n, \cB \subseteq \N^n\setminus(2\N)^n$, let $\R[\cA, \cB]$
  denote the space of all functions $f:\R^n \to \R \cup \{\infty\}$ of the form
 \begin{equation}
 \label{eq:generalfunction}
   f(\mathbf{x}) = \sum_{\alpha \in \mathcal{A}} c_{\alpha} |\mathbf{x}|^{\alpha}
     + \sum_{\beta \in \mathcal{B}} c_{\beta} \mathbf{x}^{\beta}\in\R[\cA,\cB]
 \end{equation}
 with real coefficients $c_{\alpha}$, $\alpha \in \mathcal{A}\cup\cB$. 
Our precondition $\mathcal{A} \cap \mathcal{B} \neq \emptyset$ is a 
slight restriction to the setup in \cite{knt-2019}, in order to enable
a little more convenient notation.

One motivation for the class of functions~\eqref{eq:generalfunction} is that it
allows to capture non-negativity of polynomials on $\R^n$ and non-negativity of
polynomials on the non-negative orthant $\R_+^n$ within a uniform setting. 
Moreover, global non-negativity of the summand $\sum_{\alpha \in \cA} 
c_{\alpha} |\mathbf{x}|^{\alpha}$ is equivalent 
to global non-negativity of the exponential sum 
$\yb \mapsto \sum_{\alpha \in \mathcal{A}} c_{\alpha} \exp(\alpha^T \yb)$.

 \begin{definition}
 	A function $f$ of the form~\eqref{eq:generalfunction} is called
 		an \emph{even AG function} if for at most one $\alpha\in\cA$, $c_\alpha$ is negative and for all $\beta\in\cB$, $c_\beta$ is zero; and it is called
 		an \emph{odd AG function} if for all $\alpha\in\cA$, $c_\alpha$ is non-negative and for at most one $\beta\in\cB$, $c_\beta$ is nonzero.
 		
 	$f$ is called an \emph{AG function} (\emph{arithmetic-geometric mean function}), 
 	if $f$ is an even AG function or
 	an odd AG function.
 \end{definition}

 \begin{definition}\label{de:scone}
 	Let $\emptyset \neq \cA \subseteq \R^n, \cB \subseteq \N^n\setminus(2\N)^n$ be finite disjoint sets. The \emph{$\Sc$-cone} $\CS(\cA, \cB)$ is defined as
 	\[ \CS(\cA, \cB) := \cone \left\{ f\in \R[\cA, \cB] \, : \, f \text{ is a non-negative AG function} \right\},
 	\]
 	where $\cone$ denotes the conic (or positive) hull. $\CS(\cA,\cB)$ is
 	called \emph{rational} if $\mathcal{A} \subseteq \Q^n$. 
 \end{definition}

The SAGE and SONC cones arise as special cases of this cone, see 
Section~\ref{se:prelim}.

Both from the geometric and from the optimization point of view, it is
of prominent interest to understand how the different classes of cones relate 
to each other and whether techniques for different cones can be fruitfully
combined. In \cite{karaca-2017}, Karaca, Darivianakis et al.\ have studied 
non-negativity certificates based on a
combination of the SAGE cone with the cone of sums of squares.
Concerning relations between the various cones, 
Averkov has shown that the SONC cone can be represented as a projection of
a spectrahedron \cite{averkov-2019}. In fact, his proof applies the techniques
from \cite{ben-tal-nemirovski-book}, which reveals that the 
SONC cone is even second-order representable. Wang and Magron
gave an alternative proof
based on binomial squares and $\mathcal{A}$-mediated sets
\cite{wang-magron-second-order}.

Here, we take the general view of the $\mathcal{S}$-cone as well as a
primal-dual viewpoint. Generalizing the results of Averkov and of
Wang and Magron,
we show that 
rational $\mathcal{S}$-cones and their duals are second-order representable and provide
explicit and direct descriptions. Our proof combines the techniques for the 
second-order cones techniques from \cite{ben-tal-nemirovski-book} with
the concepts and the duality theory of the $\mathcal{S}$-cone from \cite{knt-2019}.
Our derivation is different from the approach of
Wang and Magron, 
and it does not need binomial squares or $\mathcal{A}$-mediated sets. 
Moreover, our second-order representation prevents the consideration of 
redundant circuits by using a characterization of the extreme rays of the
$\mathcal{S}$-cone from \cite{knt-2019}.

Beyond the specific representability result, the goal of the paper is to offer further 
insights into the use of the framework of the $\mathcal{S}$-cone as 
a generalization of SONC and SAGE.

\medskip

\noindent
{\bf Acknowledgement.}
The work was partially supported through the 
project ``Real Algebraic Geometry and Optimization'' jointly funded by
the German Academic Exchange Service DAAD and 
the Research Council of Norway RCN.

\section{Preliminaries\label{se:prelim}}

	Throughout the text, we use the notations $\N=\{0,1,2,3,\ldots\}$ and $\R_+=\{x\in\R:x\ge 0\}$.
	For a finite subset $\cA\subseteq\R^n$, denote by $\R^\cA$ the set of $|\cA|$-dimensional vectors whose components are indexed by the set $\cA$. Moreover, we write
\[
|\xb|^{\alpha} = \prod_{j=1}^n |x_j|^{\alpha_{j}} \quad\text{ and }
\quad \xb^{\beta} = \prod_{j=1}^n x_j^{\beta_{j}},
\]
and if one component of $\xb$ is zero and the corresponding exponent is negative, then we set $|\xb|^\alpha = \infty$.

\subsection{The \texorpdfstring{$\mathcal{S}$}{S}-cone, SAGE and SONC\label{se:three-cones}}
We explain that the $\mathcal{S}$-cone generalizes the SAGE
cone and the SONC cone and collect some basic properties of
the three cones.

\smallskip

\noindent
\emph{The SAGE cone}. Let $\mathcal{A}$ be a non-empty, finite set.
An exponential sum supported on $\mathcal{A}$ is a function of the form
\begin{equation}
  \label{eq:subst-exp}
\yb \mapsto \sum_{\alpha \in \mathcal{A}} c_{\alpha} \exp(\alpha^T \mathbf{y})
\end{equation}
with real coefficients $c_{\alpha}$.
If $\mathcal{B} = \emptyset$, then $\R[\mathcal{A},\mathcal{B}]$
can be identified with the space of exponential sums supported on $\mathcal{A}$
by means of the substitution $|x_i| = \exp(y_i)$.

For
                finite $\cA\subseteq\R^n$, $\cA' \subsetneq \cA$ and
                $\beta \in \cA \setminus \cA'$, the SAGE cone 
                $C_{\mathrm{SAGE}}(\cA)$ is defined as
            \[
              C_{\mathrm{SAGE}}(\cA) = \sum_{\beta \in \cA}
                  C_{\text{AGE}}(\cA \setminus \{\beta\},\beta),
                \]
                where for $\mathcal{A'} := \mathcal{A} \setminus \{\beta\}$
                                \[
                                C_{\mathrm{AGE}}(\cA',\beta) = \Big\{c\in\R^\cA:
                                c_{\alpha} \ge 0 \text{ for } \alpha \in \cA', \,
                                \sum\limits_{\alpha\in\cA'} c_\alpha \exp(\alpha^Tx)
                                + c_\beta \exp(\beta^Tx)\ge 0
                                \text{ on } \R^n\Big\}
                                \]
(see \cite{chandrasekaran-shah-2016}). We observe that the $\mathcal{S}$-cone
$\CS(\mathcal{A},\emptyset)$ can be identified with $C_{\sage}(\mathcal{A})$ 
using the substitution~\eqref{eq:subst-exp}. $C_{\sage}(\mathcal{A})$ is 
a closed convex cone in $\R^{\mathcal{A}}$.
The membership problem for 
this convex cone can be formulated as a relative entropy program
(\cite{mcw-2018}, see also Proposition~\ref{thm:oddImplication} below).
 
\smallskip

\noindent
\emph{The SONC cone.} Here, let the non-empty finite set $\mathcal{A}$ be contained
 in $\N^n$.
 Let 
  \begin{equation} \label{eq:polynomialcircuitset}
 \begin{array}{rcl}
 I(\mathcal{A}) & = & \big\{ (A,\beta) \, : \,
 A \subseteq (2 \N)^n\cap\cA \text{ affinely independent},
 \; \beta \in \relinter(\conv A) \cap \cA \big\},
 \end{array}
 \end{equation}
 where $\relinter$ denotes the relative interior of a set.
 For singleton sets $A = \{\alpha\}$, the sets $(A,\beta)$ are formally
 of the form $(\{\alpha\}, \alpha)$. By convention, we write these circuits
 simply as $(\alpha)$, and with this convention, 
 the set $\{(\alpha) \, : \, \alpha \in (2 \N)^n\} \cap \mathcal{A}$
 is contained in $I(\cA)$.
 
 For $(A,\beta) \in I(\mathcal{A})$, let $P_{A,\beta}$ denote the set of
 polynomials  in $\R[x_1, \ldots, x_n]$ whose
 supports are contained in 
 $A\cup\{\beta\}$ 
 and which are non-negative on $\R^n$.
 The Minkowski sum
 	\[
 	C_{\sonc}(\mathcal{A}) \ = \ \sum_{(A,\beta) \ \in \  I(\mathcal{A})} P_{A,\beta}
 	\]
 	defines the \emph{cone of SONC polynomials with
 		support $\mathcal{A}$}
    (see \cite{averkov-2019,iliman-dewolff-resmathsci}).

The cone $C_{\sonc}(\mathcal{A})$ is a closed convex cone, and it can be recognized
as a special case of a rational $\mathcal{S}$-cone by observing
\[
  C_{\sonc}(\mathcal{A}) = C_{\mathcal{S}}\left( \mathcal{A}\cap(2\N)^n,\mathcal{A}\cap(\N^n\setminus(2\N)^n)\right)
\]
(see \cite{knt-2019}). Using the results from~\cite{mcw-2018}, membership
in the SONC cone can also be formulated in terms of a relative entropy program.

\smallskip

\noindent
\emph{The $\mathcal{S}$-cone.}
The $\mathcal{S}$-cone from Definition~1 offers a uniform setting for the
SAGE and the SONC cones.
We collect some further properties of the $\mathcal{S}$-cone.
For a non-empty finite set 
$\cA\subseteq \R^n$ and $\beta\in \N^n \setminus\left((2\N)^n\cup\cA\right)$ let
\[
\Po{\cA, \beta} := \left\{ f \ : \ f = \sum_{\alpha \in \cA} c_\alpha |\xb|^\alpha + c_\beta \xb^{\beta}, f(\xb) \geq 0 \; \, \forall \; \xb\in\R^n, \, c_{|\cA}\in\R_+^\cA,
\, c_\beta \in\R\right\}
\]
be the cone of non-negative odd AG functions supported on $(\cA, \beta)$, and similarly  for $\beta\in\R^n \setminus \cA$ let
\begin{equation}
\label{eq:p-a-beta-even}
\Pe{\cA, \beta} := \left\{ f \ : \ f = \sum_{\alpha \in \cA} c_\alpha |\xb|^\alpha + c_\beta
| \xb|^{\beta}, f(\xb) \geq 0 \; \, \forall \;\xb\in\R^n, \, c_{|\cA}\in\R_+^\cA,  \, 
c_{\beta} \in\R\right\}
\end{equation}
be the cone of non-negative even AG functions supported on $(\cA, \beta)$. 
By definition, 
\[
\CS(\cA, \cB) = \sum_{\alpha \in \cA} \Pe{\cA\setminus\set{\alpha}, \alpha} + \sum_{\beta \in \cB} \Po{\cA, \beta}.
\]
Note that non-negative even AG functions correspond exactly to the 
AGE functions (arithmetic-geometric exponentials) in 
\cite{chandrasekaran-shah-2016} and \cite{mcw-2018}.

\smallskip

The following alternative representation allows to express the $\mathcal{S}$-cone
in terms of the SAGE cone. Here, $|d|$ denotes the absolute value of the vector  
$d\in\R^\cB$, taken component-wise.

\begin{prop}{\cite{knt-2019}}\label{prop:SconeSAGE}
	Let $\emptyset\ne \cA\subseteq\R^n$, $\cB\subseteq\N^n\setminus(2\N)^n$ be
	finite and disjoint. 
Then,
	\begin{align} 
	\CS(\cA,\cB) & = \left\{\sum\limits_{\alpha\in\cA}c_\alpha|x|^\alpha + \sum\limits_{\beta\in\cB}d_\beta x^\beta
	: (c,-|d|)\in C_{\mathrm{SAGE}}(\cA\cup\cB)\right\} 
	\label{eq:scone-sage1} \\
	& = \left\{\sum\limits_{\alpha\in\cA}c_\alpha|x|^\alpha + \sum\limits_{\beta\in\cB}d_\beta x^\beta :
	\exists t\in\R^\cB \; \, (c,t)\in C_{\mathrm{SAGE}}(\cA\cup\cB), \,
	t\le -|d|\right\}.
	\label{eq:scone-sage2}
	\end{align}
\end{prop}

For a finite set $\emptyset \neq \cA\subseteq\R^n$,
we use the notion
\begin{align*}
	\CS(\cA):=\CS(\cA,\emptyset)=C_{\text{SAGE}}(\cA)
\end{align*}
and immediately 
observe $\CS(\cA) = \sum_{\alpha \in \cA} \Pe{\cA\setminus\set{\alpha}, \alpha}$. Hence, for our purpose it suffices to study the cone $\Pe{\cA,\beta}$ of even AG functions and use the results of this cone for the odd case in Section \ref{sec:wholeScone}.

\smallskip

Using the relative entropy function and the circuit number,
the cones $\Pe{\cA, \beta}$ and $\Po{\cA, \beta}$
can be characterized in terms of convex optimization problems.
For a finite set $\emptyset \neq \cA \subseteq \R^n$, denote by
$D:\R_{>0}^\cA \times \R_{>0}^\cA \to \R$,
\[
D(\nu, \gamma) \ = \ \sum_{\alpha \in \cA} \nu_\alpha \ln \left( \frac{\nu_\alpha}{\gamma_\alpha} \right),
\]
the \emph{relative entropy function}. $D$ can also
be extended to $\R_{+}^\cA \times \R_{+}^\cA \to \R \cup \set{\infty}$ via the conventions 
$0 \cdot \ln \frac{0}{y} = 0$ for $y \geq 0$ and $y \cdot \ln \frac{y}{0} = \infty$ for $y > 0$. 
Non-negativity of an (even or odd) AG function $f$
with coefficients $c_{\alpha}$ and $c_{\beta}$
can be characterized through the product 
$\prod_{\alpha \in \mathcal{A}} 
\left( c_\alpha / \lambda_\alpha\right)^{\lambda_\alpha}$
and $c_{\beta}$ (see \cite[Theorem 2.7]{knt-2019}). For an affinely independent ground
set $\mathcal{A}$, this product is called the \emph{circuit number} of $f$ (see
\cite{iliman-dewolff-resmathsci}). In particular, for an even AG function, this
non-negativity characterization in terms of the circuit number is given by
\begin{equation}
\label{eq:circuitnumber}
\prod_{\alpha \in \mathcal{A}} \left(\frac{c_\alpha}{\lambda_\alpha}\right)^{\lambda_\alpha} \ge -c_{\beta}. 
\end{equation}

The following characterization of $\Pe{\cA, \beta}$ and $\Po{\cA,\beta}$
in terms of the relative entropy function and in terms of
the circuit number
is a direct consequence of Theorem 2.7 of \cite{knt-2019}.

\begin{prop}\label{thm:oddImplication}
	Let $\cA\subseteq \R^n$ be a non-empty finite set, $  \beta \in \R^n \setminus\cA$ and an
	AG function
	$f$ with coefficient vector $\cb$ supported on $\cA \cup\{\beta\}$. 
	\begin{enumerate}
		\item If $f$ is an even AG function, then
		\begin{align*}
		f \in \Pe{\cA, \beta} & \iff
		\exists \nu\in\R_{+}^{\cA} \quad \sum\nolimits_{\alpha\in \cA} \nu_{\alpha}\alpha=\Big( \sum\nolimits_{\alpha\in \cA} \nu_{\alpha} \Big) \beta, \; D(\nu, e\cdot c) \le c_\beta \\
		& \iff \exists \lambda\in\R_{+}^{\cA} \quad
		\sum_{\alpha\in \cA} \lambda_{\alpha}\alpha=\beta, \; \sum\nolimits_{\alpha\in \cA} \lambda_{\alpha}=1, \; \prod_{\alpha \in \cA} \left(\frac{c_\alpha}{\lambda_\alpha}\right)^{\lambda_\alpha} \ge -c_\beta.
		\end{align*}
		\item  If $f$ is an odd AG function, then
		\begin{align*}
		f \in \Po{\cA, \beta} & \iff
		\exists \nu\in\R_{+}^{\cA} \quad \sum\nolimits_{\alpha\in \cA} \nu_{\alpha}\alpha=\Big( \sum\nolimits_{\alpha\in \cA} \nu_{\alpha} \Big) \beta, \; D(\nu, e\cdot c) \le -|c_\beta| \\
		& \iff \exists \lambda\in\R_{+}^{\cA} \quad
		\sum_{\alpha\in \cA} \lambda_{\alpha}\alpha=\beta, \; \sum\nolimits_{\alpha\in \cA} \lambda_{\alpha}=1, \; \prod_{\alpha \in \cA} \left(\frac{c_\alpha}{\lambda_\alpha}\right)^{\lambda_\alpha} \ge |c_\beta|.
		\end{align*}
	\end{enumerate}
\end{prop}
If $\cA$ is a set of affinely independent vectors and $\beta\in\relinter \cA$, 
then $\lambda$ is unique. We call the corresponding AG function a 
\emph{circuit function}, the tuple $(\cA,\beta)$ the \emph{circuit} and identify the unique $\lambda$ with the above declared characteristics 
$\lambda=\lambda(\cA,\beta)$.

\subsection{Duality theory}

Studying the duality theory has been initiated in
\cite{chandrasekaran-shah-2016} (for SAGE), 
\cite{dnt-2018} (for SONC)
and \cite{knt-2019} (for the $\mathcal{S}$-cone). See also the recent work of Papp \cite{papp-2019}, who developed an
alternative approach for deriving the dual cones, by expressing the
non-negativity of circuit polynomials in terms of a power cone.
We can identify the dual space of $\R[\cA]$ with $\R^{\cA}$.
For $f \in \R[\cA]$ with coefficients $\cb\in\R^\cA$ and an element $\vb \in \R^{\cA}$, we consider the natural duality pairing
\begin{equation}\label{eq:pairing}
\vb(f) \ = \ \sum\limits_{\alpha\in \cA}v_\alpha c_\alpha \, .
\end{equation}
Using this notation, the dual cone $(\CS({\cA}))^*$ is defined as
\[
(\CS({\cA}))^* \ = \ \left\{ \vb \in \R^{\cA} \, : \, \vb(f) \ge 0 \text{ for all } f \in \CS(\cA) \right\}.
\]

The following statement expresses the dual $\Sc$-cone in terms of
the dual SAGE cone.

\begin{prop}Let 
 $\emptyset\ne\cA\subseteq\R^n$ and $\cB\subseteq\N^n\setminus
(2\N)^n$ disjoint and finite.
	The dual cone of the $\Sc$-cone $\CS(\cA,\cB)$ is 
	\begin{align}
	\CS(\cA,\cB)^*= & \left\{(\vb,\wb)\in\R^\cA\times\R^\cB: (\vb,|\wb|)\in C_{\text{SAGE}}(\cA\cup\cB)^*\right\} 
	\label{eq:dual-scone-sage1} \\
	= & \left\{(\vb,\wb)\in\R^\cA\times\R^\cB: \exists \ub \in \R^{\cB} \; \, (\vb,\ub)\in C_{\text{SAGE}}(\cA\cup\cB)^*, \, \ub \ge |\wb|
	\right\}.
	\label{eq:dual-scone-sage2}
	\end{align}
\end{prop}

\begin{proof}
	We use~\eqref{eq:scone-sage2}, which provides
	a characterization for the primal cone $\CS(\cA,\cB)$ in
	terms of an existential quantification.
	Consider its lifted cone
	\begin{align}
	\widehat{\CS}(\cA,\cB) & :=  C_{\text{SAGE}}(\cA\cup\cB)\times\R^\cB \cap \left\{(\cb,\mathbf{t},\mathbf{d})
	: t_\beta\le -|d_\beta| \text{ for all }\beta\in\cB\right\} 
	\nonumber \\
	& = C_{\text{SAGE}}(\cA\cup\cB)\times\R^\cB \cap \left\{(\cb,\mathbf{t},\mathbf{d})
	: t_\beta\le d_\beta, t_\beta\le -d_\beta \text{ for all }\beta\in\cB\right\}
	\label{eq:liftedcone1}
	\end{align}
	in the space $\R^\cA\times\R^\cB\times\R^\cB$.
	The dual cone of the right-hand cone in~\eqref{eq:liftedcone1} is
	the set
	\begin{align*}
	\cone \left\{(0,\ldots,0,-e^{(\beta)}, \pm e^{(\beta)})
	\, : \,
	 \beta\in\cB\right\},
	\end{align*}
	where $e^{(\beta)}$ denotes the unit vector with respect to $\beta\in\cB$.
	As intersection and Minkowski sum are dual operations, we obtain
	\begin{align*}
	\widehat{\CS}(\cA,\cB)^*= C_{\text{SAGE}}(\cA\cup\cB)^*\times
	\{0\} + \cone \left\{(0,\ldots,0,-e^{(\beta)},\pm e^{(\beta)})
	\, : \,
	\beta\in\cB \right\}.
	\end{align*}
	
	Identifying the $\mathcal{S}$-cone with its coefficients, we can express
	$\CS(\cA,\cB)^*$ in terms of the lifted cone $\widehat{\CS}(\cA,\cB)$ by
	\[
	  \CS(\cA,\cB)^* = 
	  \widehat{\CS}(\cA,\cB) \cap \left\{ 
	  (\vb,\mathbf{s},\wb) \in \R^{\cA} \times \R^{\cB} \times \R^{\cB} \, : \,
	  \mathbf{s} = 0\right\}.
	\]
	Thus, $(\vb,\wb) \in \CS(\cA,\cB)$ whenever
	$(\vb,|\wb|) \in C_{\text{SAGE}}(\cA \cup \cB)^*$.
	Convexity then implies the 
	second characterization~\eqref{eq:dual-scone-sage2}.
\end{proof}

Hence, as in the primal case, it suffices to study even 
AG functions in the dual situation. We will make use of a representation 
of the dual of the $\Sc$-cone from 
\cite{knt-2019}. For this, observe that similar to the SONC case
in~\eqref{eq:polynomialcircuitset}, one can also consider
circuits in the case of the SAGE cone. In slight variation
of~\eqref{eq:polynomialcircuitset}, for a finite set
 $\emptyset\ne\cA\subseteq\R^n$,
the set of \emph{circuits} supported on $\cA$ is the set
\[
\begin{array}{rcl}
I(\mathcal{A}) & = & \big\{ (A,\beta) \, : \,
A \subseteq \cA \text{ affinely independent}, \; \beta \in \relinter(\conv A) \cap (\cA\setminus A) \big\}.
\end{array}
\]
 
 Two examples of circuits are the pairs $(A,\beta)$ with $A=\{0,6\}$ and $\beta=\{2\}$ (see Figure \ref{fig:circuit1}) and $(A',\beta')$ with $A'=\{(0,0)^T,(4,2)^T,(2,4)^T\}$ and $\beta'=(1,1)^T$ (see
 Figure \ref{fig:circuit2}).

\begin{figure}
\begin{minipage}{.5\textwidth}
	\centering
\begin{tikzpicture} [scale=1.2,cap=round]
\def\costhirty{0.8660256}

\tikzstyle{axes}=[]
\tikzstyle{important line}=[very thick]
\tikzstyle{information text}=[rounded corners,fill=purple!10,inner sep=1ex]

\draw[style=help lines,step=0.5cm] (0,0) grid (3.2,0);

\begin{scope}[style=axes]
\draw[->] (-0.25,0) -- (3.25,0) node[right] {$x$};

\end{scope}

\draw[style=important line,cyan]
(0,0) -- (3,0);

\draw[style=important line,cyan]
(3,0.05) -- (3,-0.05) node[below=2pt] {$6$};

\draw[style=important line,cyan]
(0,0.05) -- (0,-0.05) node[below=2pt] {$0$};

\draw[style=important line,purple]
(1,-.05) -- (1,-.05) node[below=2pt] {$2$} ;

\draw[style=important line,purple]
(.95,-.05) -- (1.05,.05);

\draw[style=important line,purple]
(1.05,-.05) -- (.95,.05) ;

\draw[draw=none,white]
(-0.25,0) -- (-0.25,0)  ;
\draw[draw=none,white]
(2.25,0) -- (2.25,0) node[right] {$x$} ;
\draw[draw=none,white]
(0,-0.25) -- (0,-0.25) ;
\draw[draw=none,white]
(0,2.25) -- (0,2.25) node[above] {$y$} ;

\end{tikzpicture}
\caption{Circuit $(A,\beta)$} \label{fig:circuit1}
\end{minipage}%
\begin{minipage}{.5\textwidth}
	\centering
		\begin{tikzpicture} [scale=1,cap=round]
		
		\tikzstyle{axes}=[]
		\tikzstyle{important line}=[very thick]
		\tikzstyle{information text}=[rounded corners,fill=red!10,inner sep=1ex]
		
		\draw[style=help lines,step=0.5cm] (0,0) grid (2.2,2.2);
		
		\begin{scope}[style=axes]
		\draw[->] (-0.25,0) -- (2.25,0) node[right] {$x$};
		\draw[->] (0,-0.25) -- (0,2.25) node[above] {$y$};
		
		\end{scope}
		
		\draw[style=important line,purple]
		(0,0) -- (2,1) node[right=2pt] {$(4,2)^T$};
		
		\draw[style=important line,purple]
		(1,2) -- (0,0) node[below=2pt] {$(0,0)^T$};
		
		\draw[style=important line,purple]
		(1,2) node[above=2pt] {$(2,4)^T$} -- (2,1) ;
		
		\draw[style=important line,cyan]
		(.45,.45) -- (.55,.55) ;
		
		\draw[style=important line,cyan]
		(.45,.55) -- (.55,.45) ;
		
		\draw[draw=none,cyan]
		(0,0.5) node[left=2pt] {$(1,1)^T$} -- (0,0.5) ;

		\end{tikzpicture}
	\caption{Circuit $(A',\beta')$} \label{fig:circuit2}
\end{minipage}
\end{figure}

Thereby, the dual $\Sc$-cone $\CS(\cA)$ can be represented as follows.

\begin{prop}\label{prop:equiv}\cite[Theorem 3.5]{knt-2019}
	Let $\emptyset \neq \cA \subseteq \R^n$ be finite. Then a point 
	$\vb \in \R^{\cA}$
	is contained in $\CS(\cA)^*$ if and only if $\vb\ge 0$ and
	\begin{align*} 
		\ln (v_\beta) \leq \sum_{\alpha\in A} \lambda_{\alpha} \ln(v_\alpha) 
   \text{ for every circuit } (A, \beta) \text{ in } I(\cA)
		\text{ and } \lambda=\lambda(A,\beta).
  \end{align*} 
\end{prop}

\subsection{Second-order formulations}
Let $[m]$ abbreviate the set $\{1, \ldots, m\}$ and denote
by $\|\cdot\|$ the Euclidean norm.
A \emph{second-order cone program} \emph{(SOCP)} is an optimization problem 
of the form
\begin{equation}
  \label{eq:SOCP}
	\min \left\{ \cb^T \mathbf{x} \, : \,
   ||A_i \xb +\bb_i||_2 \le \cb_i^T \bb +\db_i \text{ for all }i\in[m] \right\}
\end{equation}
with real symmetric matrices $A_i$,
vectors $\bb_i,\cb_i,\db_i$ and a vector $\cb$.
A subset of $\R^n$ is called \emph{second-order representable} if it can be represented
as a projection of the feasible set of a second-order program.

For a symmetric
$2 \times 2$-matrix, positive semidefiniteness can be formulated as a
second-order condition.

\begin{lemma}{\rm (See, e.g., \cite[\S 6.4.3.8]{nesterov-nemirovski},
\cite[Lemma 4.3]{wang-magron-second-order}.)}\label{lem:SDP-SOCP}
	A symmetric $2\times 2$ matrix $A=\left(\begin{array}{cc} a & b\\ b & c
	\end{array}\right)$ is positive semidefinite if and only if the second-order condition
	\[
	 \left|\left|\left(\begin{array}{c} 2b \\ a-c
		\end{array}\right) \right|\right|_2 \le a+c
	\]
	is satisfied.
\end{lemma}

Let $S_+^n$ be the subset of symmetric $n \times n$-matrices which are 
positive semidefinite.
By \cite{averkov-2019}, there exists some $m\in\N$ so that the cone of SONC polynomials $C_{\text{SONC}}(\mathcal{A})$ supported on
$\mathcal{A}$ can be written as the projection of the spectrahedron $(S_+^2)^m\cap H$ for some affine space $H$. 
	
\section{A second-order representation for the cone of non-negative AG functions and its dual\label{se:agfunctions}}

In order to provide a second-order representation for the $\mathcal{S}$-cone and its
dual, the main task is to capture the cone of non-negative AG functions and its dual.
For a comprehensive collection of techniques for handling second-order cones,
we refer to \cite{ben-tal-nemirovski-book}.

Throughout the section,
let $(A,\beta)$ be a fixed circuit and rational
barycentric coordinates $\lambda\in\R_+^A$,
which represent $\beta$ as a convex combination of $A$. That is, 
$\beta = \sum_{\alpha \in \mathcal{A}} \lambda_{\alpha} \alpha$ and $\sum_{\alpha \in \mathcal{A}} \lambda_{\alpha} =1$.
Let $p\in\N$ denote the smallest common denominator of the fractions $\lambda_\alpha$ for $\alpha\in A$, i.e., $\lambda_\alpha=\frac{p_\alpha}{p}$ with $p_\alpha\in\N$
for all $\alpha\in A$ and $p$ is minimal.
 
With the given circuit $(A,\beta)\in I(\cA)$, we associate a set of \emph{dual circuit variables}
	\begin{align}\label{eq:circuitvars}
	(y_{k,i})_{k,i},
	\end{align}
where $k\in [\lceil \log_2(p)\rceil -1]$ and 
$i\in [2^{\lceil \log_2(p)\rceil-k}]$. 
The collection of these $\sum_{k=1}^{\lceil \log_2(p)\rceil-1} 
2^{\lceil \log_2(p)\rceil-k} $ $= 2^{\lceil \log_2(p)\rceil}-2$ variables is
denoted as $\yb^{A,\beta}$ or shortly as $\yb$. Further, denote
the restriction of a vector $\vb \in \R^{\cA}$ to the components of 
$A \subseteq \cA$ by $\vb_{|A}$.

\begin{definition}\label{def_circuitmatrix}
 A \emph{dual circuit matrix} $C_{A,\beta}^*(\vb_{| A},v_\beta,\yb)$ is a 
 block diagonal matrix consisting of the blocks 
	\begin{align}\label{type1}
	\left(\begin{array}{cc}
y_{k-1,2i-1} & y_{k,i}  \\
	y_{k,i} & y_{k-1,2i} 
	\end{array}\right)
	\quad 
	\text{for $k\in \{2,\ldots,\lceil \log_2(p)\rceil-1 \}$ and $i\in [2^{\lceil \log_2(p)\rceil-k}]$},
	\end{align}
	\begin{equation}\label{specialtypeeven}
	\left(\begin{array}{cc}
	y_{\lceil \log_2(p)\rceil-1,1} &  v_\beta \\
	v_\beta & y_{\lceil \log_2(p)\rceil-1,2}
	\end{array}\right),
	\end{equation}
	the singleton block
	$		
	( v_\beta ),
	$
	as well as $2^{\lceil \log_2(p)\rceil-1}$ blocks of the form
	\begin{align}\label{type2}
	\left(\begin{array}{cc}
     u & y_{1,l}  \\
	y_{1,l} &	w 
	\end{array}\right)
	\quad
	\text{ for } l\in [2^{\lceil \log_2(p)\rceil-1} ],
	\end{align}
	where in each of these blocks $u$ and $w$ represent a variable of the set $\{v_\alpha \, : \, \alpha\in A\}\cup\{v_\beta\}$
	such that altogether each $v_{\alpha}$ appears $p_{\alpha}$ times and
	$v_\beta$ appears $2^{\lceil \log_2(p)\rceil}-p$ times. 
\end{definition}

In this definition, the exact order of appearances of the variables in
$\{v_\alpha \, : \, \alpha\in A\}\cup\{v_\beta\}$ is not uniquely determined.
However, since this order of appearances will not matter, we will
speak of \emph{the} dual circuit matrix.

\begin{rem}
Each block of the type~\eqref{type2} contains two (not necessarily identical)
variables from the set $\{v_\alpha \, : \, \alpha\in A\}\cup\{v_\beta\}$.
Since $\sum_{\alpha \in A} \lambda_{\alpha} = 1$, we have 
$\sum_{\alpha \in A} p_{\alpha} = p$ and hence the total number
of occurrences of variables from the set $\{v_\alpha \, : \, \alpha\in A\}\cup\{v_\beta\}$
in the blocks of type~\eqref{type2} is
\[
  \sum_{\alpha \in A} p_{\alpha} + (2^{\lceil \log_2(p)\rceil}-p) 
  =  2^{\lceil \log_2(p)\rceil},
\]
which is twice the number of blocks of type~\eqref{type2}.
\end{rem}

Note that 
every $y_{k,i}$ only serves as an auxiliary variable to make the non-linear 
constraints
$
\ln (v_\beta) \le \sum\nolimits_{\alpha\in A} \lambda_\alpha \ln(v_\alpha ) 
$
of the dual $\Sc$-cone description from Proposition~\ref{prop:equiv}
linear. In the end, we will only multiply those 
constraints to obtain the original ones.
In particular, factors $v_\beta$ serve to cover cases where 
$p$ is not a power of $2$.
For the purpose of the second-order descriptions, it does 
not matter in which order the variables appear in the blocks~\eqref{type2},
because only the product of these blocks will be considered.

The goal of this subsection is to show the following characterization of the cone of non-negative even AG functions $\Pe{A,\beta}$ supported on the circuit $(A,\beta)$.
Here, positive semi\-defi\-nite\-ness of a symmetric matrix is denoted by 
$\succeq 0$.

\begin{thm} \label{thm:projDualCirc}
	The dual cone $(\Pe{A,\beta})^*$ of the cone of non-negative even AG functions 
	$\Pe{A,\beta}$ supported on the circuit $(A,\beta)\in I(\cA)$ 
	is the projection of the spectrahedron
	\begin{align}
	  \label{eq:spectrahedron-dual}
	&	\left\{(\vb, \yb) \in \R^{\cA}\times \R^{2^{\lceil \log_2(p)\rceil}-2} \ : \ C_{A,\beta}^*(\vb_{| A},v_\beta, \yb) \succcurlyeq 0 \right\}
	\end{align}
	 on $(\vb_{|A}, v_\beta)$.
	 $(\Pe{A,\beta})^*$ is second-order representable.
\end{thm}

Here, 
the second-order representability follows immediately from the 
representation~\eqref{eq:spectrahedron-dual} in connection with 
Lemma~\ref{lem:SDP-SOCP}.
Let us consider an example for the theorem.

\begin{example}
	Let $\cA=\{0,6\}, \cB=\{2\}$ and consider the circuit $(A,\beta)$ with $A=\cA$ and $\beta=2$ (compare Figure \ref{fig:circuit1}). We have $p=3, p_0=2, p_6=1$ and $\yb$
	consists of the components
	\[
	\ y_{1,1}, \ y_{1,2}. 
	\]
	A vector $(v_0,v_2,v_6)$ is contained 
	in $(\Pe{A,\beta})^*$ if and only if $v_2\ge 0$ and the three
	$2 \times 2$-matrices
	\[	\left(\begin{array}{cc}
	y_{1,1} & v_2 \\
	v_2 &	y_{1,2} 
	\end{array} \right), \; \left(\begin{array}{cc}
	v_0 & y_{1,1} \\
	y_{1,1} & v_0 
	\end{array} \right), \;
	\left(\begin{array}{cc}
	v_6 & y_{1,2} \\
	y_{1,2} & v_2  
	\end{array}\right)
	\]
	are positive semidefinite.
\end{example}

In \cite{averkov-2019}, Averkov considered the size of the blocks in the SDP-representation of SONC-polynomials but does not give a number or bound
on the number of blocks. Here, for the $\mathcal{S}$-cone,
we provide a bound on the number of inequalities of a second-order 
representation, which also gives a bound on the number of 
$2 \times 2$-blocks in a semidefinite representation. The bound
depends on the smallest common denominator of the barycentric 
coordinates representing the inner exponent of a circuit as a convex 
combination of the outer ones.

\begin{cor} \label{le:sizedual}
The matrix $C_{A,\beta}^*(\vb_{|A},v_\beta,\yb)$ consists of
${2^{\lceil \log_2(p)\rceil}}-1$ blocks of size $2 \times 2$
and one block of size $1 \times 1$.
\end{cor}

\begin{proof}
Counting the number of $2 \times 2$-blocks,
there are
$\sum_{k=2}^{\lceil \log_2(p)\rceil-1}\left(2^{\lceil \log_2(p)\rceil-k} \right)
  =  2^{\lceil \log_2(p)\rceil-1}$ $-2$ blocks of type~\eqref{type1}, a single block (\ref{specialtypeeven}) and 
  $2^{\lceil \log_2(p)\rceil-1}$
  blocks of type~\eqref{type2}.
\end{proof}
  
\begin{rem}\label{rem:ineq}
It is useful to record the set inequalities characterizing the
positive semidefiniteness of the matrix $C_{A,\beta}^*(\vb_{|A},v_\beta,\yb)$. Besides the non-negativity conditions
for the variables,
\begin{align}
		\vb_{|A} & \ge 0, 
        \quad 
		v_\beta \ge 0,\label{ineq2}\\
		\text{ and } x_{k,i} & \ge 0 
 \text{ for all } k\in \{2,\ldots,\lceil \log_2(p)\rceil-1\}, i\in [2^{\lceil 
 \log_2(p)\rceil}-k],\label{ineq3}
\end{align}
these are the determinantal conditions arising from the positive semidefiniteness
of the matrices in~\eqref{type1}, \eqref{specialtypeeven} and \eqref{type2}:
\begin{align}
		v_\beta^2 & \le  {y_{\lceil \log_2(p)\rceil-1,1}y_{\lceil \log_2(p)\rceil-1,2}},\label{ineq4}\\
		y_{k,i}^2& \le y_{k-1,2i-1} y_{k-1,2i} 
\text{ for all }k\in \{2,\ldots,\lceil \log_2(p)\rceil-1\},
i\in [2^{\lceil \log_2(p)\rceil-k}] \label{ineq6}\\
			\text{ and } uw &\ge 
\left(y_{1,l}\right)^2 \label{ineq7} \text{ for }  l\in [2^{\lceil \log_2(p)\rceil-1}]
			\end{align}
			for $u,w\in \{v_\alpha \, : \, \alpha\in A\}\cup
			\{v_\beta\}$, such that $v_{\alpha}$  appears 
			$p_{\alpha}$ times for every $\alpha\in A$ 
			and $v_\beta$ appears $2^{\lceil \log_2(p)\rceil}-p$ times.
\end{rem}

The next lemma prepares one inclusion of Theorem \ref{thm:projDualCirc}.

\begin{lemma}\label{lem:2inclusion}
Let $\vb \in\R^{A,\beta}$ such that there exists $\yb \in \R^{2^{\lceil \log_2(p)\rceil }-2}$ 
with $C_{A,\beta}^*(\vb_{|A},v_\beta, \yb) \succcurlyeq 0$.
Then $\vb_{|A}$ is non-negative and satisfies
	\begin{align*}
		v_\beta^p \le \prod\limits_{\alpha\in A} v_\alpha^{p_\alpha}.
	\end{align*}
\end{lemma}
\begin{proof}
  By~\eqref{ineq2}, we have $\vb_{|A} \ge 0$ and
  $v_{\beta} \ge 0$.
Moreover, \eqref{ineq4} and successively applying~\eqref{ineq6} gives
	 \begin{eqnarray*}
	 v_\beta &  \le & \left( y_{\lceil \log_2(p)\rceil-1,1} \, y_{\lceil \log_2(p)\rceil-1,2}\right)^{1/2} \\
	 &	\le & \left( y_{\lceil \log_2(p)\rceil-2,1} \, y_{\lceil \log_2(p)\rceil-2,2} \right)^{1/4}
	  \left( y_{\lceil \log_2(p)\rceil-2,3} \, y_{\lceil \log_2(p)\rceil-2,4}\right)^{1/4} \\ [0.5ex]
	 & = & \left(y_{\lceil \log_2(p)\rceil-2,1} \, 
	 y_{\lceil \log_2(p)\rceil-2,2} \, 
	 y_{\lceil \log_2(p)\rceil-2,3} \, 
	 y_{\lceil \log_2(p)\rceil-2,4}\right)^{\frac{1}{2^{\lceil \log_2(p)\rceil-(\lceil \log_2(p)\rceil-2)}}}\\ [1ex]
	 & \le & \cdots 
	 \le  \left(\left(\prod\nolimits_{\alpha\in A} v_\alpha^{p_\alpha}\right)\cdot \left(v_\beta\right)^{2^{\lceil \log_2(p)\rceil}-p} \right)^{\frac{1}{2^{\lceil \log_2(p)\rceil}}}.
	 \end{eqnarray*}
	 This is equivalent to 
	 \begin{align*}
	 	\left(v_\beta\right)^{2^{\lceil \log_2(p)\rceil}} \cdot \left(v_\beta\right)^{p-2^{\lceil \log_2(p)\rceil}} \le \prod\nolimits_{\alpha\in A} v_\alpha^{p_\alpha},
	 \end{align*}
	 which implies
	 $
	 v_\beta^p \le \prod_{\alpha \in A} v_\alpha^{p_\alpha}.
	 $
\end{proof}

Now we prepare the converse inclusion of Theorem~\ref{thm:projDualCirc}.

\begin{lemma} \label{lem:1inclusion}
	For every $\vb \in\R^{A,\beta}$ with $\vb_{|A \cup \{\beta\}} \ge 0$ and $v_\beta^p \le \prod_{\alpha\in A} v_\alpha^{p_\alpha}$, there exists $\yb \in\R^{2^{\lceil \log_2(p)\rceil}-2}$ 
such that $C_{A,\beta}^*(\vb_{|A},v_\beta, \yb) \succcurlyeq 0$. 
\end{lemma}

\begin{proof}
 Define $\yb$ inductively by
 \begin{align*}
 & y_{1,l}=\sqrt{uw} \text{ for those } u,w \text{ which occur in the block with }y_{1,l}, \\
 & 	y_{k,i}=\sqrt{y_{k-1,2i-1}y_{k-1,2i}} \text{ for all }k\in \{2,\ldots,\lceil \log_2(p)\rceil-1\}, i\in [2^{\lceil \log_2(p)\rceil-k}]. 
 \end{align*}
 It suffices to show that the inequalities~\eqref{ineq2}-\eqref{ineq7} in 
 Remark~\ref{rem:ineq} are satisfied.
 The non-negativity conditions~\eqref{ineq2} and~\eqref{ineq3} 
 hold by assumption and by definition of $\yb$.
 The construction of $\yb$ also implies that
 a subchain of the chain of inequalities considered in the previous proof
 even holds with equality,
 	 \begin{eqnarray*}
	 && \left( y_{\lceil \log_2(p)\rceil-1,1} \, y_{\lceil \log_2(p)\rceil-1,2}\right)^{1/2} \\
	 &	= & \left( y_{\lceil \log_2(p)\rceil-2,1} \, y_{\lceil \log_2(p)\rceil-2,2} \right)^{1/4}
	  \left( y_{\lceil \log_2(p)\rceil-2,3} \, y_{\lceil \log_2(p)\rceil-2,4}\right)^{1/4} \\ [0.5ex]
	 & = & \left(y_{\lceil \log_2(p)\rceil-2,1} \, 
	 y_{\lceil \log_2(p)\rceil-2,2} \, 
	 y_{\lceil \log_2(p)\rceil-2,3} \, 
	 y_{\lceil \log_2(p)\rceil-2,4}\right)^{\frac{1}{2^{\lceil \log_2(p)\rceil-(\lceil \log_2(p)\rceil-2)}}}\\ [1ex]
	 & = & \cdots 
	 = \left(\left(\prod\nolimits_{\alpha\in A} v_\alpha^{p_\alpha}\right)\cdot \left(v_\beta\right)^{2^{\lceil \log_2(p)\rceil}-p} \right)^{\frac{1}{2^{\lceil \log_2(p)\rceil}}}.
	 \end{eqnarray*}
  By the assumption $v_\beta^p \le \prod_{\alpha\in A} v_\alpha^{p_\alpha}$,
  we obtain 
  $v_\beta^2 \le  {y_{\lceil \log_2(p)\rceil-1,1}y_{\lceil \log_2(p)\rceil-1,2}}$,
  which shows inequality~\eqref{ineq4}.
 The remaining inequalities~\eqref{ineq6}, \eqref{ineq7} 
 are satisfied with equality by construction.
\end{proof}

Finally, we can conclude the proof of Theorem \ref{thm:projDualCirc}. 
\begin{proof}[Proof of Theorem \ref{thm:projDualCirc}]
	Let $p$ be defined as in Definition \ref{def_circuitmatrix} and $\lambda\in\R^A$ denote the barycentric coordinates representing $\beta$ as a convex combination of $A$, i.e., $\lambda_\alpha=\frac{p_\alpha}{p}$ with $p_\alpha\in \N$ for all $\alpha\in A$. By~\eqref{eq:p-a-beta-even} and Proposition~\ref{prop:equiv}, we have
	\begin{align*}
		(\Pe{A,\beta})^* & =\left\{\vb\in\R^{A,\beta} \, : \, \vb_{|A\cup\{\beta\}} \ge 0, \;\ln(v_\beta)\le \sum\nolimits_{\alpha\in A} \lambda_\alpha \ln(v_\alpha )\right\}\\
		& = \left\{\vb\in\R^{A,\beta} \, : \, \vb_{|A\cup\{\beta\}}\ge 0, \; v_\beta^p \le \prod\nolimits_{\alpha\in A} v_\alpha^{p_\alpha}\right\}.
	\end{align*}
Applying Lemmas \ref{lem:2inclusion} and \ref{lem:1inclusion}, we obtain that $C^*_{A,\beta}(x,v_\beta)\succcurlyeq 0$ if and only if $\vb \in P_{A,\beta}^*$.
	\end{proof}
	
	Our derivation of the second-order representation of the dual cone
	$(\Pe{A,\beta})^*$ also suggests a simple way to
	derive a second-order cone representation of the primal 
	cone $\Pe{A,\beta}$.
	For the dual cone, Proposition~\ref{prop:equiv} gives -- besides 
	non-negativity-constraints on $v_\alpha$ for $\alpha\in\cA$ and on $v_{\beta}$ --  
	the condition
	$
		\ln(v_\beta) \le \sum\nolimits_{\alpha\in A} \lambda_\alpha \ln(v_\alpha ) 
     $
    for every circuit 
    $(A,\beta)\in I(\cA)$.
	Those conditions can -- as done in the previous proof -- be stated as 
	\begin{align*}
	v_\beta^p \le \prod_{\alpha\in A} v_\alpha^{p_\alpha},
	\text{ where }
	\lambda_\alpha=\frac{p_\alpha}{p} \, .	
	\end{align*}
	
	The conditions for the primal cone can be reformulated similarly. 
	Namely, by~\eqref{eq:circuitnumber},
	an even circuit function $f$ with
	coefficient vector $\cb$ is non-negative if and only if
	$
	    -c_{\beta} \le 
	    \prod_{\alpha \in A} \left( c_{\alpha} / \lambda_{\alpha}  \right)^{\lambda_{\alpha}},
	$
	which we write as
	\begin{align*}
	 (-c_\beta)^p \le \prod_{\alpha\in A} \left(\frac{c_\alpha}{\lambda_\alpha}\right)^{p_\alpha}.
	\end{align*}
	
	 This motivates to carry over the definition of the dual circuit matrix to the
	 primal case as follows. Since $c_\beta$ may be negative (in contrast to the 
	 dual case), we introduce 
  the \emph{primal circuit variables}, or simply \emph{circuit variables},
	 	\begin{align*}
	 		(x_\beta, 
	 		(x_{k,i})_{k,i}),
	 	\end{align*}
	 	where $k\in [\lceil \log_2(p)\rceil]$ and $i\in [2^{\lceil \log_2(p)\rceil -k}]$. As in the dual case, we refer to these $1+\sum_{k=1}^{\lceil\log_2(p)\rceil}2^{\lceil\log_2(p)\rceil -k}=2^{\lceil \log_2(p)\rceil}$ variables as $\xb^{A,\beta}$ or shortly as $\xb$.
	 
	\begin{definition}[Circuit matrix]
		The \emph{circuit matrix} $C_{A,\beta}(\cb_{|A\cup\{\beta\}},x_\beta,\xb)$ is 
		the block diagonal matrix consisting of the blocks 
			\begin{align*}
			& \left(\begin{array}{cc}
			x_{k-1,2i-1} & x_{k,i}   \\
			x_{k,i} & x_{k-1,2i} 
			\end{array}\right)
			\quad \text{ for } k\in \{2,\ldots, \lceil \log_2(p)\rceil \}, \ i\in [2^{\lceil \log_2(p)\rceil-k}], \end{align*}
			the two singleton blocks
			\begin{align} \label{eq:onebyone2} \left(\begin{array}{c}
			x_{\lceil \log_2(p)\rceil,1} -	\left(\prod\nolimits_{\alpha\in A} (\lambda_\alpha)^{\lambda_\alpha}\right)x_\beta \end{array}\right), 
			\quad \left(\begin{array}{c}
			x_\beta+c_\beta
			\end{array}\right),
			\end{align}
			as well as $2^{\lceil \log_2(p)\rceil-1}$ blocks of the form
			\begin{align}
			\label{eq:primalblocks3}
			\left(\begin{array}{cc}
			u & x_{1,l}  \\
			x_{1,l} &	 w 
			\end{array}\right) \quad \text{ for } l\in [2^{\lceil \log_2(p)\rceil-1} ],
			\end{align}
			where $u,w \in \{c_\alpha \, : \, \alpha\in A\}\cup\{\left(\prod\nolimits_{\alpha\in A} (\lambda_\alpha)^{\lambda_\alpha}\right)x_\beta\}$, 
such that $c_{\alpha}$ appears $p_{\alpha}$ times for every $\alpha\in A$ and $\left(\prod\nolimits_{\alpha\in A} (\lambda_\alpha)^{\lambda_\alpha}\right)x_\beta$ appears $2^{\lceil \log_2(p)\rceil}-p$ times. 
	\end{definition}
	
	Note that for a circuit $(A,\beta)$, the product $\left(\prod\nolimits_{\alpha\in A} (\lambda_\alpha)^{\lambda_\alpha}\right)$ is always non-zero, because 
	$\beta \in \relinter \conv A$ and $A$
	consists of affinely independent vectors.
	
 In contrast to the dual cone, there is no sign constraint on $c_\beta$ in the
 primal cone. If $p$ is not a power of $2$, then
$x_\beta$ appears on the main diagonal of~\eqref{eq:primalblocks3}. In our
coupling of $x_\beta$ with $c_{\beta}$, the constraint $x_\beta +c_\beta \ge 0$ 
results in $-c_\beta\le x_\beta$ and thus reflects these sign considerations.
	
Note that the primal cone consists of circuit \emph{functions}, whereas in our
definition of the dual cone, the elements are coefficient vectors.
Therefore, the projection regarded in Theorem \ref{thm:projDualCirc} only delivers 
the coefficients of the circuit functions rather than the cone itself.
	
\begin{thm} \label{thm:projCirc}
	The set of coefficients of the cone $\Pe{A,\beta}$ of non-negative even
	circuit polynomials supported on the circuit $(A,\beta)$ coincides with the projection
    of the spectrahedron
	\begin{align}
	  \label{eq:primal-proj}
	& \widehat{\Pe{A,\beta}} :=	\left\{(\cb, \xb) \in \R^{\cA}\times \R^{2^{\lceil \log_2(p)\rceil }} \, : \, C_{A,\beta}(\cb_{|A\cup\{\beta\}},x_\beta, \xb) \succcurlyeq 0 ,
	\ c_{|\cA\setminus \left(A\cup\{\beta\}\right)}=0\right\}
	\end{align}
	on $(\cb_{|A},c_\beta)$.
	The cone $\Pe{A,\beta}$ is second-order representable.
\end{thm}

The last equality constraint in~\eqref{eq:primal-proj} is redundant and
can be omitted. We include it here, because this formulation is needed
in Section \ref{sec:wholeScone} for the description of the $\Sc$-cone supported on the full set 
$\cA$.

\begin{proof} 
First, let $(\cb,\xb)\in\widehat{\Pe{A,\beta}}$.
The positive semidefiniteness of the $2\times 2$-blocks in $C_{A,\beta}(\cb_{|A\cup\{\beta\}}, $ $x_\beta, \xb)$ imply the inequalities
		\begin{align*}
		  \cb_{|A} \ge 0 \text{ and }  (-x_\beta)^p \cdot \left(\prod\nolimits_{\alpha\in A} {\lambda_\alpha}^{\lambda_\alpha}\right)
		\le \prod\nolimits_{\alpha\in A} c_\alpha^{p_\alpha}.
		\end{align*}
		The two $1\times 1$-blocks from~\eqref{eq:onebyone2}
		give the inequalities
		$
			x_{\lceil \log_2(p)\rceil,1} \ge 	\left(\prod\nolimits_{\alpha\in A} \lambda_\alpha^{\lambda_\alpha}\right)x_\beta\text{ and }
			x_\beta \ge -c_\beta.
		$
		They imply
		$
			-c_\beta\left(\prod\nolimits_{\alpha\in A} \lambda_\alpha^{\lambda_\alpha}\right) \le x_\beta \left(\prod\nolimits_{\alpha\in A} \lambda_\alpha^{\lambda_\alpha}\right) \le  	x_{\lceil \log_2(p)\rceil,1}.
		$ Hence, similar to Lemma \ref{lem:2inclusion},
		 \begin{eqnarray*}
			&& x_\beta \left(\prod\nolimits_{\alpha\in A} \lambda_\alpha^{\lambda_\alpha}\right)\le  	x_{\lceil \log_2(p)\rceil,1}\ \le \ \left( x_{\lceil \log_2(p)\rceil-1,1} \, x_{\lceil \log_2(p)\rceil-1,2}\right)^{1/2} \\
			&	\le & \left( x_{\lceil \log_2(p)\rceil-2,1} \, x_{\lceil \log_2(p)\rceil-2,2} \right)^{1/4}
			\left( x_{\lceil \log_2(p)\rceil-2,3} \, x_{\lceil \log_2(p)\rceil-2,4}\right)^{1/4} \\ [0.5ex]
			& = & \left(x_{\lceil \log_2(p)\rceil-2,1} \, 
			x_{\lceil \log_2(p)\rceil-2,2} \, 
			x_{\lceil \log_2(p)\rceil-2,3} \, 
			x_{\lceil \log_2(p)\rceil-2,4}\right)^{\frac{1}{2^{\lceil \log_2(p)\rceil-(\lceil \log_2(p)\rceil-2)}}}\\ [1ex]
			& \le & \cdots 
			\le  \left(\left(\prod\nolimits_{\alpha\in A} c_\alpha^{p_\alpha}\right)\cdot \left(x_\beta\right)^{2^{\lceil \log_2(p)\rceil}-p}\left(\prod\nolimits_{\alpha\in A} \lambda_\alpha^{\lambda_\alpha}\right)^{2^{\lceil \log_2(p)\rceil}-p} \right)^{\frac{1}{2^{\lceil \log_2(p)\rceil}}}.
		\end{eqnarray*}
		This is equivalent to 
		\begin{align*}
		\left(x_\beta\right)^{2^{\lceil \log_2(p)\rceil}}\cdot \left(\prod\nolimits_{\alpha\in A} \lambda_\alpha^{\lambda_\alpha}\right)^{2^{\lceil \log_2(p)\rceil}} \cdot \left(x_\beta\right)^{p-2^{\lceil \log_2(p)\rceil}}\cdot \left(\prod\nolimits_{\alpha\in A} \lambda_\alpha^{\lambda_\alpha}\right)^{p-2^{\lceil \log_2(p)\rceil}} \le \prod\nolimits_{\alpha\in A} c_\alpha^{p_\alpha},
		\end{align*}
		which, together with the considerations before the chain of inequalities, yields
		$
		(-c_\beta)^p \le \prod_{\alpha \in A} (c_\alpha/\lambda_{\alpha})^{p_\alpha}
		$
		and further $\cb_{|A \cup \{\beta\}} \in \Pe{A,\beta}$.
		
		For the converse inclusion, we remind the reader that $\lambda_\alpha >0$ for all $\alpha \in A$. We set $x_\beta := x_{\lceil \log_2(p)\rceil,1}\left(\prod\nolimits_{\alpha\in A} \left(\frac{1}{\lambda_\alpha}\right)^{\lambda_\alpha}\right) $ and, 
		similar to the proof of Lemma~\ref{lem:1inclusion},
		 define $\xb$ inductively by
		\begin{align*}
		& x_{1,l}=\sqrt{uw} \text{ for those } u,w \text{ which occur in the block with }x_{1,l}, \\
		& 	x_{k,i}=\sqrt{x_{k-1,2i-1}x_{k-1,2i}} \text{ for all }k\in \{2,\ldots,\lceil \log_2(p)\rceil\}, i\in [2^{\lceil \log_2(p)\rceil-k}].
		\end{align*}
   Analogous to that proof, the construction of $\xb$
    gives $C_{A,\beta}(\cb_{A \cup \{\beta\}},x_{\beta},\xb) \succeq 0$.
    
     Second-order
	representability is then an immediate consequence in view of
	Lemma~\ref{lem:SDP-SOCP}.
\end{proof}

\begin{example}
	Let $\cA=\{0,2\}$, $\cB=\{1\}$ and consider the circuit $(A,\beta)$ with $A=\cA$ and $\beta=1$. Since
	\[
	1=\frac{1}{2}\cdot 0+ \frac{1}{2}\cdot 2,
	\]
	we have $p_1=p_2=1$ and $p=2$. 
	Hence, $\lceil \log_2(p) \rceil = \log_2(p)= 1$, 
	$2^{\lceil \log_2(p) \rceil}-p=2-p=0$ as well as
	\begin{align*}
	\prod\limits_{\alpha\in A}  \lambda_\alpha^{\lambda_\alpha} =\frac{1}{2} 
	\; \text{ and } \; \xb=\left(\begin{array}{c}
	x_1 \\
	x_{1,1}
	\end{array}\right).
	\end{align*}
	A given vector $(c_0,c_1,c_2)$ is contained in $P_{\mathcal{A},\beta}$
	if and only if  
	\begin{align*}
	x_{1,1}  - \frac{1}{2}x_1 \ge 0, \;
	x_{1}+c_1 \ge 0
	\; \text{ and } \;
	 \left(\begin{array}{cc}
	  c_0 & x_{1,1} \\
	  x_{1,1} & c_2 
	\end{array}\right) \succeq 0.
	\end{align*}
\end{example}

Similar to Lemma \ref{le:sizedual}, we can determine the number of 
blocks.

\begin{cor} The matrix $C_{A,\beta}(\cb_{|A\cup \{\beta\}},x_\beta,\xb)$ consists of ${2^{\lceil \log_2(p)\rceil}}-1$ blocks of size $2\times 2$ and two blocks of size $1\times 1$.
\end{cor}
	
\section{A second-order representation of the \texorpdfstring{$\Scal$}{S}-cone and its dual}\label{sec:wholeScone}

In Section~\ref{se:agfunctions}, 
we obtained second-order representations of the subcones of non-negative even 
circuit functions and their duals, under the condition that the barycentric coordinates
are rational. We now assume that $\mathcal{A}$ and $\mathcal{B}$ are
rational and derive an explicit second-order representation of the rational
$\Scal$-cone $\CS(\cA,\cB)$ and its dual. In the primal case, those cones are 
obtained via projection and Minkowski sum, and in the dual case, they arise from
projection and intersection. First we consider the lifted cones for the dual case.

Taking all circuits $(A,\beta)$ into account would induce a highly redundant
representation. To avoid those redundancies, we make use of the 
following characterization from \cite{knt-2019} of the extreme rays of the 
$\mathcal{S}$-cone.

For finite and disjoint sets $\emptyset\ne\cA, \cB\subseteq\R^n$, the set of \emph{reduced} circuits contained in $\cA\cup\cB$ is the set

\begin{eqnarray*}
R(\mathcal{A},\cB) & = & \big\{ (A,\beta) \, : \,
A \subseteq \cA \text{ affinely independent}, \; \; \beta \in \relinter(\conv A) \cap (\cB\setminus A), \\
& & \; \cA\cap(\conv(A))\setminus(A\cup\{\beta\})=\emptyset \big\}.
\end{eqnarray*}

Less formally, this is the set of all circuits with outer exponents in $\cA$ and inner exponents in $\cB$ without additional support points contained in the convex hull of $A$. 

Note that for $\cA\subseteq\R^n$ and $\cB\subseteq\N^n\setminus(2\N)^n$ disjoint and finite, the set $R(\cA,\cA)$ is exactly the set of \emph{even reduced circuits} and the set $R(\cA,\cB)$ the set of \emph{odd reduced circuits}. The set $R(\cA,\cA\cup\cB)$ denotes the set of \emph{all} reduced circuits $(A,\beta)$ with $A\subseteq\cA$ and $ \beta\in\cA\cup\cB$. A circuit function supported on a reduced circuit in 
$R(\cA,\cA \cup \cB)$ has non-negative coefficients corresponding to exponents in $\cA$ and a possibly negative coefficient corresponding to a single exponent in $\cA\cup\cB$.

The question whether a circuit is reduced or not depends on the ground set $\cA$. For example, the circuit $(A,\beta)$ with $A=\left\{\left(\begin{array}{c}
0 \\ 0
\end{array}\right),\left(\begin{array}{c}
4 \\ 0
\end{array}\right), \left(\begin{array}{c}
0 \\ 2
\end{array}\right) \right\}$ and $\beta=\left(\begin{array}{c}
1 \\ 1
\end{array}\right)$ is \emph{reduced} for the ground set $\cA=A\cup\{\beta\}\cup\left\{{\left(\begin{array}{c}
	4 \\ 2
	\end{array}\right)} \right\}$ (compare Figure \ref{fig:reduced}), but 	\emph{not reduced} for $\cA=A\cup\{\beta\}\cup\left\{{\left(\begin{array}{c}
	{2} \\ {0}
	\end{array}\right)} \right\}$ (compare Figure \ref{fig:nonreduced}).
\begin{figure}
	\begin{minipage}{.5\textwidth}
		\centering
		\begin{tikzpicture}[scale=1.0,cap=round]
		
		\tikzstyle{axes}=[]
		\tikzstyle{important line}=[very thick]
		\tikzstyle{information text}=[rounded corners,fill=red!10,inner sep=1ex]
		
		\draw[style=help lines,step=0.5cm] (0,0) grid (2.2,1.2);
		
		\begin{scope}[style=axes]
		\draw[->] (-0.25,0) -- (2.25,0) node[right] {$x$};
		\draw[->] (0,-0.25) -- (0,1.25) node[above] {$y$};
		
		\end{scope}
		
		\draw[style=important line,purple]
		(0,0) -- (2,0) node[below=2pt] {$(4,0)^T$};
		
		\draw[style=important line,purple]
		(0,1) -- (0,0) node[below=2pt] {$(0,0)^T$};
		
		\draw[style=important line,purple]
		(0,1) node[left=2pt] {$(0,2)^T$} -- (2,0) ;
		
		\draw[style=important line,cyan]
		(.45,.45) -- (.55,.55) ;
		
		\draw[style=important line,cyan]
		(.45,.55) -- (.55,.45) ;
		
		\draw[style=important line,blue]
		(1.95,.95) -- (2.05,1.05) ;
		
		\draw[style=important line,blue]
		(1.95,1.05) -- (2.05,.95) ;
		
		\draw[draw=none,cyan]
		(4,1.5) node[left=2pt] {$(1,1)^T$} -- (4,1.5) ;
		
				\draw[draw=none,blue]
				(5.5,1.5) node[left=2pt] {$(4,2)^T$} -- (5.5,1.5) ;
		
		\end{tikzpicture}
		\caption{The circuit is reduced, as $(4,2)^T\notin\conv(A)$.} \label{fig:reduced}
	\end{minipage}%
	\begin{minipage}{.5\textwidth}
		\centering
		\begin{tikzpicture}[scale=1,cap=round]
		
		\tikzstyle{axes}=[]
		\tikzstyle{important line}=[very thick]
		\tikzstyle{information text}=[rounded corners,fill=red!10,inner sep=1ex]
		
		\draw[style=help lines,step=0.5cm] (0,0) grid (2.2,1.2);
		
		\begin{scope}[style=axes]
		\draw[->] (-0.25,0) -- (2.25,0) node[right] {$x$};
		\draw[->] (0,-0.25) -- (0,1.25) node[above] {$y$};
		
		\end{scope}
		
		\draw[style=important line,purple]
		(0,0) -- (2,0) node[below=2pt] {$(4,0)^T$};
		
		\draw[style=important line,purple]
		(0,1) -- (0,0) node[below=2pt] {$(0,0)^T$};
		
		\draw[style=important line,purple]
		(0,1) node[left=2pt] {$(0,2)^T$} -- (2,0) ;
		
		\draw[style=important line,cyan]
		(.45,.45) -- (.55,.55) ;
		
		\draw[style=important line,cyan]
		(.45,.55) -- (.55,.45) ;
		
		\draw[style=important line,blue]
		(0.95,.05) -- (1.05,-.05) ;
		
		\draw[style=important line,blue]
		(1.05,.05) -- (0.95,-.05) ;
		
		\draw[draw=none,cyan]
		(4,1.5) node[left=2pt] {$(1,1)^T$} -- (4,1.5) ;
		
				\draw[draw=none,blue]
				(5.5,1.5) node[left=2pt] {$(2,0)^T$} -- (5.5,1.5) ;

		\end{tikzpicture}
		\caption{The circuit is not reduced, as $(2,0)^T\in\conv(A)$.} \label{fig:nonreduced}
	\end{minipage}
\end{figure}

The following proposition is a direct consequence of Theorem 3.5(d) 
in~\cite{knt-2019}.

\begin{prop} Let $\emptyset\ne\cA\subseteq\R^n$ and $\cB\subseteq\N^n\setminus(2\N)^n$ be finite and disjoint sets. Then
	$$\CS(\cA,\cB)=\sum\limits_{(A,\beta)\in R(\cA,\cA)} \Pe{A,\beta} + \sum\limits_{(A,\beta)\in R(\cA,\cB)} \Po{A,\beta}.$$
\end{prop}

Using this decomposition
theorem, we can exclude many circuits from our consideration.
Thus, the second-order program will be much smaller than the 
one considering all circuits.

In Section $3$, we only considered even circuits. To use Lemma~\ref{prop:SconeSAGE} and obtain the conditions for odd circuits as well, we extend the dual circuit variables for odd circuits to
\begin{align*}
	(y_\beta,(y_{k,i})_{k,i})
\end{align*}
for $k\in[2^{\lceil \log_2(p)\rceil} -1]$ and $i\in[2^{\log_2(p)-k}]$. We call them $\yb^{A,\beta}$ nevertheless for a fixed circuit $(A,\beta)\in R(\cA,\cB)$.

For the dual case, we consider the coordinates
\[
\yb^{\mathcal{A},\mathcal{B}} = 
\left\{ (\yb^{A,\beta}) \, : \, (A,\beta)\in R(\cA,\cA\cup\cB) \right\},
\]
which consist of $\sum_{(A,\beta)\in R(\cA,\cA\cup\cB)} 2^{\lceil \log_2(p_{A,\beta})\rceil}-1$ components,
where $p_{A,\beta}$ denotes the smallest common 
denominator of the barycentric coordinates
 $\lambda_{A,\beta}$ of the circuit $(A,\beta)$ representing 
$\beta$ as a convex combination of $A$.

For the primal case, we consider
\[
  \xb^{\mathcal{A},\mathcal{B}} = 
   \left\{ (\xb^{A,\beta}) \, : \, (A,\beta)\in R(\cA,\cA\cup\cB) \right\},
\]
which consist of $\sum_{(A,\beta)\in R(\cA,\cA\cup\cB)} 2^{\lceil \log_2(p_{A,\beta})\rceil}$ components. 

Using Lemma~\ref{prop:SconeSAGE}, we can use our earlier characterizations
of $\Pe{A,\beta}$ to obtain the following second-order characterization
for $\Po{A,\beta}$.

\begin{cor}
	Let $(A,\beta)\in R(\cA,\cB)$ an odd reduced circuit with rational 
	$A\subseteq\cA \subseteq \Q^n$ and $\beta\in\cB$.
	\begin{itemize}
		\item[(1)] Let $f$ be an odd AG function supported on $(A,\beta)$ with coefficient vector $\cb$. $f$ is non-negative if and only if there exists $\xb\in \R^{2^{\lceil\log_2(p) \rceil }}$ such that $C_{A,\beta}(\cb_{|A},x_\beta, \xb) \succcurlyeq 0$ and 
		\begin{align}\label{eq:retransformationScone}
		\left(\begin{array}{cc}
		x_\beta & c_\beta \\
		c_\beta & x_\beta
		\end{array}\right)\succcurlyeq 0.
		\end{align}
		\item[(2)] A vector $\vb\in\R^{A,\beta}$ is contained in 
		$\left(\Po{A,\beta}\right)^*$ if and only if there exist $\yb\in \R^{2^{\lceil\log_2(p) \rceil }-2}$ and $y_\beta\in\R$ such that $C_{A,\beta}^*(\vb_{|A},y_\beta, \yb) \succcurlyeq 0$ and
		\begin{align}\label{eq:retransformationDualScone}
		\left(\begin{array}{cc}
		y_\beta & v_\beta \\
		v_\beta & y_\beta
		\end{array}\right)\succcurlyeq 0.
		\end{align}
	\end{itemize}
\end{cor}

Note that, as a consequence of the application of Lemma~\ref{prop:SconeSAGE},
the second argument of $C_{A,\beta}^*(\vb_{|A},y_\beta, \yb)$ is $y_\beta$ now instead of $v_\beta$ that we had in Theorem~\ref{thm:projDualCirc}.

\begin{proof}
(1) The semidefinite condition on the
    matrix~\eqref{eq:retransformationScone} is equivalent to
		$
			x_\beta\ge 0 \text{ and } |c_\beta| \le x_\beta.
		$
		Hence, altogether we obtain
		\begin{align*}
			f\in\Po{A,\beta} \ \text{ if and only if } \ |c_\beta| \le \prod\limits_{\alpha\in A}\left(\frac{c_\alpha}{\lambda}\right)^{\lambda_\alpha}
		\end{align*}
		for barycentric coordinates $\lambda\in\R_+^A$ decomposing $\beta$ as a convex combination of $A$. This is exactly
	Proposition~\ref{thm:oddImplication}(b).
	
(2) If $\vb \in (\Po{A,\beta})^*$, then, in 
 the notation of Theorem~\ref{eq:dual-scone-sage2}, 
 there exists some $u$ such that 
$(\vb,u) \in (\Pe{A,\beta})^*$ and $u\ge |v_\beta|$. 
In particular, $u\ge 0$ is necessary for containment in $\left(\Pe{A,\beta}\right)^*$. The semidefinite 
constraints~\eqref{eq:retransformationDualScone}
		are equivalent to $y_\beta\ge 0$ and the latter inequality $u\ge |v_\beta|$, and the constraint $C_{A,\beta}^*(\vb_{|A},y_\beta, \yb) \succcurlyeq 0$ is equivalent to $(\vb,y_\beta)\in\left(\Pe{A,\beta}\right)^*$ by Theorem~\ref{thm:projDualCirc}.
\end{proof}

For every \emph{odd} reduced circuit $(A,\beta)\in R(\cA,\cB)$,
define the block diagonal matrix $\widehat C_{A,\beta}^*(\vb_{|A\cup\{\beta\}},y_\beta,\yb)$ consisting of the dual circuit matrix $C_{A,\beta}^*(\vb_{|A\cup\{\beta\}},y_\beta,\yb)$ and (\ref{eq:retransformationScone}) for the dual cone.
Considering all the reduced circuits, these lifting matrices 
define the lifted cone
\begin{align*}
  \widehat{C}^*(\cA,\cB) =  \big\{ (\vb, \yb^{\mathcal{A},\mathcal{B}}) \ : \
	& \widehat C_{A,\beta}^*(\vb_{|A\cup\{\beta\}},y_\beta, \yb) \succcurlyeq 0 \text{ for all } (A,\beta)\in R(\cA,\cB),\\
 & 	 C_{A,\beta}^*(\vb_{|A},v_\beta, \yb) \succcurlyeq 0 \text{ for all } (A,\beta)\in R(\cA,\cA) \big\},
	\end{align*}
	where the variable vector $\vb$ lives in the space $\R^{\mathcal{A},\mathcal{B}}$.

	For a fixed odd reduced circuit $({A},{\beta})\in R(\cA,\cB)$, let
	\begin{align*}
		\widehat{\Po{A,\beta}}=
		\big\{(\cb,\xb^{\mathcal{A},\mathcal{B}}) \, : \, 
	\widehat 	C_{{A},{\beta}}(\cb_{|{A}\cup\{\beta\}},x_{{\beta}},\xb^{{A},{\beta}}) \succcurlyeq 0 , c_{|\cA\cup\cB\setminus(A\cup\{\beta\})}=0\big\},
	\end{align*}
	where $\widehat 	C_{{A},{\beta}}(\cb_{|{A}\cup\{\beta\}},x_{{\beta}},\xb^{{A},{\beta}})$ is defined analogous to the dual case. We define the lifted cone
	\begin{align*}
		\widehat{C}(\cA,\cB)=\sum\limits_{({A},{\beta})\in R(\cA,\cA)} \widehat{\Pe{{A},{\beta}}} + \sum\limits_{({A},{\beta})\in R(\cA,\cB)} \widehat{\Po{{A},{\beta}}}.
	\end{align*}
	Here, for every $(A,\beta)\in R(\cA,\cA)$, $\widehat{\Pe{{A},{\beta}}} $ is the set from Theorem \ref{thm:projCirc}.

\begin{cor}
	\begin{itemize}
		\item[(1)] The dual of the rational $\Scal$-cone $\CS^*(\cA,\cB)$ is the projection on the coordinates $\vb\in \R^{\cA,\cB}$ of $\widehat{C}^*(\cA,\cB)$.
		\item[(2)] The primal rational $\Sc$-cone $\CS(\cA,\cB)$ is the projection on the coordinates
		$\vb \in \R^{\cA,\cB}$ of $\widehat{C}(\cA,\cB)$.
	\end{itemize}
\end{cor}

Applying this lifting to the second-order representations of 
Theorems~\ref{thm:projCirc} and~\ref{thm:projDualCirc} in standard form
also gives second-order representations of 
$C_{\mathcal{S}}(\mathcal{A},\mathcal{B})$
and
$C^*_{\mathcal{S}}(\mathcal{A},\mathcal{B})$
in standard form.

\begin{cor}[Second-order representation of the dual rational $\Sc$-cone]
  \label{co:dual1}
	A vector $\vb \in\R^{(\cA,\cB)}$ is contained in the rational $\Sc$-cone
	$(C_{\mathcal{S}}(\mathcal{A},\mathcal{B}))^*$ if and only if the circuit vector 
	$\yb^{\cA,\cB}$
	satisfies 
	for every reduced odd circuit $(A,\beta)\in R(\cA, \cB)$
	
	\begin{enumerate}
	\item
		$
		\left(\begin{array}{cc}
		y^{A,\beta}_{k-1,2i-1} & y^{A,\beta}_{k,i} \\
		y^{A,\beta}_{k,i} & y^{A,\beta}_{k-1,2i}
	\end{array}\right) \succeq 0,
	\quad 2 \le k \le\lceil \log_2(p_{A,\beta})\rceil-1 \; 
	\forall i\in [2^{\lceil \log_2(p_{A,\beta})\rceil-k}], \smallskip
    $
    \item
    $ 	\left(\begin{array}{cc}
    y^{A,\beta}_{\lceil \log_2(p_{A,\beta})\rceil -1,1} & y^{A,\beta}_\beta \\
    y^{A,\beta}_\beta & y^{A,\beta}_{\lceil \log_2(p_{A,\beta})\rceil -1,2}
    \end{array}\right) \succeq 0,
	$ \smallskip
	\item
	$
	\left(\begin{array}{c@{\:\:}c} 
	u & y^{A,\beta}_{1,l} \\
	y^{A,\beta}_{1,l} & w 
	\end{array} \right) \succeq 0	
	$
	for $l\in [2^{\lceil \log_2(p_{A,\beta})\rceil-1} ]$ and
	$u,w\in \{v_\alpha : \alpha\in A\}\cup\{y^{A,\beta}_\beta\}$, such that $v_{\alpha}$ appears $(p_{A,\beta})_{\alpha}$ times for each $\alpha\in A$ and $y^{A,\beta}_\beta$ appears $2^{\lceil \log_2(p_{A,\beta})\rceil}-(p_{A,\beta})_{\alpha}$ times,
	\item  
	$
	\left|\left|v_\beta \right|\right|_2 \le y^{A,\beta}_\beta, 
	$
		
	\end{enumerate}
	and for every reduced even circuit $(A,\beta)\in R(\cA, \cA)$ the conditions of Theorem~\ref{thm:projDualCirc}.
\end{cor}

We need to write $\yb^{A,\beta}$ instead of just writing $\yb$ in the previous corollary,
since different $\yb^{A,\beta}$ for every reduced circuit $(A,\beta)$ may appear.

For the primal case, we have to consider every reduced circuit as well. Here, sums
take the role of the intersections from the dual case.

\begin{cor}[A second-order representation of the rational $\Sc$-cone]
  \label{co:primal1}
     A function $f \in \R[\mathcal{A},\mathcal{B}]$ with coefficient vector 
     $\mathbf{c}$ is contained in the rational $\Sc$-cone 
	$C_{\mathcal{S}}(\mathcal{A}, \mathcal{B})$ if and only if there exists
	$\cb^{A,\beta}$ for $(A,\beta)\in R(\cA,\cA\cup\cB)$ 
	with
     $\cb=\sum\limits_{(A,\beta)\in R(\cA,\cA\cup\cB)}\cb^{A,\beta}$ and for the circuit vector $\xb^{\cA,\cB}$
	and for every $(A,\beta)\in R(\cA,\cA\cup\cB)$ the following inequalities hold.
	\begin{enumerate}
		\item
	$ \left(\begin{array}{cc} x^{A,\beta}_{k-1,2i-1} & x^{A,\beta}_{k,i} \\ x^{A,\beta}_{k,i} & x^{A,\beta}_{k-1,2i}
		\end{array}\right) \succcurlyeq 0 ,
		\; 2 \le k \le \lceil \log_2(p_{A,\beta})\rceil, \; i \in [2^{\lceil \log_2(p_{A,\beta})\rceil-k}],$ 
  \item
		$x^{A,\beta}_{\lceil \log_2(p_{A,\beta})\rceil,1} -\left(\prod\nolimits_{\alpha\in A} \lambda_\alpha^{(p{A,\beta})_\alpha}\right)x^{A,\beta}_\beta \ge 0,$
  \item $x^{A,\beta}_\beta +c_\beta \ge 0$,
	\item 
	$
	\left|\left|c_\beta \right|\right|_2 \le x^{A,\beta}_\beta
	$ if $(A,\beta)$ is an odd circuit,
	\item as well as in both the even and the odd case,
	\begin{align*}
	 \left(\begin{array}{cc} u & x^{A,\beta}_{1,l} \\ 	x^{A,\beta}_{1,l} & w
	\end{array}\right) \succcurlyeq 0
	\quad \text{ for } l\in [2^{\lceil \log_2(\lambda_{A,\beta})\rceil-1} ]
	\end{align*}
	for $u,w \in \{c_\alpha \, : \, \alpha\in A\}\cup\big\{\big(\prod\nolimits_{\alpha\in A} \lambda_\alpha^{(\lambda_{A,\beta})_\alpha}\big)x^{A,\beta}_\beta\big\}$, 
	such that $c_{\alpha}$ appears $(p_{A,\beta})_{\alpha}$ times for every $\alpha\in A$ and $\big(\prod\nolimits_{\alpha\in A} \lambda_\alpha^{(\lambda_{A,\beta})_\alpha}\big)x^{A,\beta}_\beta$ appears $2^{\lceil \log_2(p_{A,\beta})\rceil}-p_{A,\beta}$ times. 
	\end{enumerate}
\end{cor}

As already mentioned in Section~\ref{se:prelim}, the SONC cone 
$C_{\sonc}(\mathcal{A})$ and its dual are always rational $\mathcal{S}$-cones
and thus occur as a special case of Corollaries~\ref{co:primal1} and
\ref{co:dual1}.

\begin{rem} 
The specific case of the primal SONC cone has also been studied in detail by
Magron and Wang \cite{wang-magron-second-order}. Their approach is based on different methods.
In particular, it
	relies on mediated sets and intermediately
	uses sums of squares representations. However, the resulting
	second-order programs are structurally similar. Notably, the
	dependence of the size of the second-order
	program on the parameter $p$ in our derivation relates to the
	dependency on the size of the rational mediated set in \cite{wang-magron-second-order}.
	Note also that various amendments are integrated into the
	approaches (such as the handling of denominators in
	\cite{wang-magron-second-order} and the use of extreme rays in
	our approach).
\end{rem}

\smallskip

\noindent
{\bf Acknowledgment.} We thank an anonymous referee for some beneficial
suggestions.

\section{Conclusion and open question}

We have provided second-order representations for primal and dual 
rational $\mathcal{S}$-cones. These statements remain valid 
also for non-rational sets $\mathcal{A}$, as long as all the relevant
barycentric coordinates are still rational. 
It is an open question whether an $\mathcal{S}$-cone and its dual are
also second-order representable in the general non-rational case.

Also, despite the use of the reduced circuits, the 
second-order representation of the $\mathcal{S}$-cone is still rather
large. It remains the question whether smaller
second-order representations for the $\mathcal{S}$-cone exist.

\bibliography{bibSoncSage_FULL}
\bibliographystyle{plain}

\end{document}